\documentclass{amsart}
\usepackage{graphicx}
\usepackage{amssymb,amscd,amsthm,amsxtra}
\usepackage{latexsym}
\usepackage{epsfig}

\newtheorem{thm}{Theorem}[section]
\newtheorem{cor}[thm]{Corollary}
\newtheorem{lem}[thm]{Lemma}
\newtheorem{prop}[thm]{Proposition}
\theoremstyle{definition}
\newtheorem{defn}[thm]{Definition}
\theoremstyle{remark}
\newtheorem{rem}[thm]{Remark}
\numberwithin{equation}{section}

\newcommand{\R}{\mathbb R}

\newcommand{\eps}{\epsilon}

\newcommand{\p}{\partial}
\newcommand{\di}{\displaystyle}

\newcommand{\comment}[1]{}

\begin{document}

\title[Regularity  of flat free boundaries]{A one-phase problem for the fractional Laplacian: regularity of flat free boundaries.}
\author{D. De Silva}
\address{Department of Mathematics, Barnard College, Columbia University, New York, NY 10027}
\email{\tt  desilva@math.columbia.edu}
\author{O. Savin}
\address{Department of Mathematics, Columbia University, New York, NY 10027}
\email{\tt  savin@math.columbia.edu}
\author{Y. Sire}
\address{LATP, UMR CNRS 7353, Universit\'e Aix-Marseille, France}
\email{\tt sire@cmi.univ-mrs.fr}
\begin{abstract} We consider a one-phase free boundary problem involving a fractional Laplacian $(-\Delta)^\alpha$, $0<\alpha <1,$  and we prove that ``flat free boundaries" are $C^{1,\gamma}$. We thus extend the known result for the case $\alpha=1/2.$\end{abstract}
\maketitle

\section{Introduction}

In the last decade, a large amount of work has been devoted to non linear equations involving non local operators with special attention for the so-called fractional laplacian $(-\Delta)^\alpha$, where $\alpha \in (0,1)$. This is a Fourier multiplier in $\mathbb R^n$ whose symbol is $|\xi|^{2\alpha}$. The main feature of this operator is its non locality, which can be seen from the alternative definition given by its integral representation (see \cite{L}) 
$$(-\Delta)^\alpha u (x) = PV \int_{\R^n} \frac{u(x) - u(y)}{|x-y|^{n+2\alpha}} dy$$
where $PV$ denotes the Cauchy principal value (up to a renormalizing constant depending on $n$ and $\alpha$.)

This paper investigates the regularity properties of a free boundary problem involving the fractional Laplacian. More precisely, we are interested in a Bernoulli-type one-phase problem. The classical one is given by 
\begin{equation}\label{AC}\begin{cases}
\Delta u = 0, \quad \textrm{in $\Omega \cap \{u > 0 \} ,$}\\
|\nabla u|= 1, \quad \textrm{on  $\Omega \cap  \p \{u>0\},$} 
\end{cases}\end{equation}
with $\Omega$ a domain in $\R^n.$
A pioneering investigation of \eqref{AC} was that of Alt and Caffarelli \cite{AC} (variational context), and then Caffarelli  \cite{C1,C2,C3} (viscosity solutions context).

As a natural generalization of \eqref{AC}, we consider the following problem (see for instance the book \cite{DL})

\begin{equation}\label{ACalpha}\begin{cases}
(-\Delta)^\alpha u = 0, \quad \textrm{in $\Omega \cap \{u > 0 \} ,$}\\
\di\lim_{t \rightarrow 0^+} \dfrac{u
(x_0+t\nu(x_0))}{t^\alpha} = const., \quad \textrm{on  $ \Omega \cap \p\{u>0\},$} 
\end{cases}\end{equation}
with $u$ defined on the whole $\R^n$ with prescribed values outside of $\Omega$.  This problem has been first investigated by Caffarelli, Roquejoffre and the third author in \cite{CafRS}. 

The non locality of the fractional Laplacian makes computations hard to handle directly on the equation. However by a result by Caffarelli and Silvestre \cite{CSi}, one can realize it as a boundary operator in one more dimension. More precisely, given $\alpha\in(0,1)$ and a function $u \in H^\alpha(\mathbb R^n)$ we consider the minimizer $g$ to
\begin{equation} \label{argmin} 
{\rm min}\left\{ \int_{\mathbb R^{n+1}_{+}}z^\beta \left| \nabla g \right|^2\;dx dz \; : \;  
g|_{\partial\mathbb R^{n+1}_{+}}=u\right\}
\end{equation} 
with $$\beta:=1-2\alpha \in(-1,1).$$

The ``extension" $g$ solves the Dirichlet problem
\begin{equation*}\label{bdyFrac2} 
\left \{
\begin{aligned} 
\textrm{div\,} (z^\beta \nabla g)&=0 \qquad 
{\mbox{ in $\mathbb R^{n+1}_+$}} 
\\
g&= u  
\qquad{\mbox{ on $\partial\mathbb R^{n+1}_+$,}}\end{aligned}\right . 
\end{equation*} 
and $(-\triangle)^\alpha u$ is a Dirichlet to Neumann type operator for $g$. Precisely in \cite{CSi} it is shown that
$$(-\Delta)^\alpha u= 
- d_\alpha \displaystyle{\lim_{z \rightarrow 0^+}} z^{\beta} \partial_z g,$$
where $d_\alpha$ is a positive constant depending only on $n$ and~$\alpha$,
and the equality holds in the distributional sense.

Due to the variational structure of the extension problem, one can consider the following functional, associated to \eqref{ACalpha}, 

 $$J(g,B_1) = \int_{B_1} |z|^\beta|\nabla g|^2 dx dz + \mathcal{L}_{\R^n} (\{g>0\} \cap \R^n \cap B_1).$$ 
The minimizers of $J$ have been investigated in \cite{CafRS}, where general properties (optimal regularity, nondegeneracy,
classification of global solutions), corresponding to those proved in \cite{AC} for the classical Bernoulli problem \eqref{AC}, have been obtained. In \cite{CafRS}, only a partial result concerning the regularity of the free boundary is obtained. The question of the regularity of the free boundary in the case $\alpha=1/2$ was subsequently settled in a series of papers co-authored by the first and the second author of this note \cite{DR, DS1, DS2}. 

In this paper, in view of the previous discussion, we consider the following thin one-phase problem associated to the extension

\begin{equation}\label{FBintro}\begin{cases}
\text{div}(|z|^\beta \nabla g) = 0, \quad \textrm{in $B_1^+(g):= B_1 \setminus \{(x,0) : g(x,0) = 0 \} ,$}\\
\dfrac{\p g}{\p t^\alpha}= 1, \quad \textrm{on  $F(g):=  \p_{\R^n}\{x \in \mathcal{B}_1 : g(x,0)>0\} \cap \mathcal{B}_1,$} 
\end{cases}\end{equation}
where $\beta=1-2\alpha,$  \begin{equation}\label{nabla_U}
\dfrac{\p g}{\p t^\alpha}(x_0): = \di\lim_{t \rightarrow 0^+} \frac{g(x_0+t\nu(x_0))} {t^\alpha},  \quad x_0 \in F(g) \end{equation} and  $\mathcal{B}_r \subset \R^n$ is the $n$-dimensional ball of radius $r$ (centered at 0).

A special class of viscosity solutions to \eqref{FBintro} (with the constant 1 replaced by a precise constant $A$ depending on $n$ and $\alpha$) is provided by minimizers of the functional $J$ above.

We explain below the free boundary condition \eqref{nabla_U}. In Section 2 we show that in the case $n=1$, a particular 2-dimensional solution $U(t,z)$ to our free boundary problem is given by 
\begin{equation}\label{U}U = \left (r^{1/2}\cos \frac \theta 2 \right )^{2\alpha}, \end{equation}
with $r$, $\theta$ the polar coordinates in the $(t,z)$ plane. This function is simply the ``extension" of $(t^+)^\alpha$ to the upper half-plane, reflected evenly across $z=0$. By boundary Harnack estimate (see Theorem \ref{bhi}), any solution $g$ to $$div (|z|^\beta \nabla g) = 0, \quad \mbox{in} \quad \R^2 \setminus \{(t,0)| t \le 0 \}$$ that vanishes on the negative $t$ axis satisfies the following expansion near the origin
$$g=U(a+o(1)),$$
for some constant $a$. Then $\frac{\p g} { \p t^\alpha}(0)=a$ and the constant $a$ can be thought as a ``normal" derivative of $g$ at the origin. 

The 2-dimensional solution $U$ describes also the general behavior of $g$ near the free boundary $F(g)$. 
Indeed,
in the $n$-dimensional case, if $0 \in F(g)$ and $F(g)$ is $C^2$ then the same expansion as above holds in the 2-dimensional plane perpendicular to $F(g)$ at the origin. We often denote the limit in \eqref{nabla_U} as $\p g / \p U$ and it represents the first coefficient of $U$ in the expansion of $g$ as above.

We now state our main result  about the regularity of $F(g)$ under appropriate flatness assumptions (for all the relevant definitions see Section 2).

\begin{thm} \label{mainT}There exists a small constant $\bar \eps >0$ depending on $n$ and $\alpha$, such that if $g$ is a viscosity solution to \eqref{FBintro}  satisfying
\begin{equation}\label{Flat2} \{x \in \mathcal{B}_1 : x_n \leq -\bar \eps\} \subset \{x \in \mathcal{B}_1 : g(x,0)=0\} \subset \{x \in \mathcal{B}_1 : x_n \leq \bar \eps \},\end{equation} then $F(g)$ is $C^{1,\gamma}$ in $\mathcal{B}_{1/2}$, with $\gamma>0$ depending on $n$ and $\alpha$.
\end{thm}

The previous theorem has the following corollary.

\begin{cor} \label{mainC}There exists a universal constant $\bar \eps >0$, such that if $u$ is a viscosity solution to \eqref{ACalpha} in $\mathcal B_1$  satisfying
\begin{equation*} \{x \in \mathcal{B}_1 : x_n \leq -\bar \eps\} \subset \{x \in \mathcal{B}_1 : u(x,0)=0\} \subset \{x \in \mathcal{B}_1 : x_n \leq \bar \eps \},\end{equation*} then $F(u)$ is $C^{1,\gamma}$ in $\mathcal{B}_{1/2}$.
\end{cor}

The Theorem above extends the results in \cite{DR} to any power $0<\alpha<1.$ We follow the strategy developed in \cite{DR}. Most of the proofs remain valid in this context as well, since they rely on basic facts such as Harnack Inequality, Boundary Harnack inequality, Comparison Principle and elementary properties of $U$.

The paper is organized as follows. In section 2 we introduce notation, definitions and preliminary results. In Section 3 we recall the notion of $\eps$- domain variations and the corresponding linearized problem. Section 4 is devoted to Harnack inequality while Section 5 contains the proof of the main improvement of flatness theorem. In Section 6 the regularity of the linearized problem is investigated.

\section{Preliminaries}

In this Section we introduce notation, definitions, and preliminary results.

\subsection{Notation}
 
A point $X \in \R^{n+1}$ will be denoted by $X= (x,z) \in \R^n \times \R$. We will also use the notation $x=(x',x_n)$ with $x'=(x_1,\ldots, x_{n-1}).$ A ball in $\R^{n+1}$ with radius $r$ and center $X$ is denoted by $B_r(X)$ and for simplicity $B_r = B_r(0)$. Also we use $\mathcal{B}_r$ to denote the $n$-dimensional ball $B_r \cap \{z=0\}$. 

Let $v(X)$ be a continuous non-negative function in $B_1$. We associate to $v$ the following sets: \begin{align*}
& B_1^+(v) := B_1 \setminus \{(x,0) : v(x,0) = 0 \} \subset \R^{n+1};\\
& \mathcal{B}_1^+(v):= B_1^+(v) \cap \mathcal{B}_1 \subset \R^{n};\\
& F(v) := \p_{\R^n}\mathcal{B}_1^+(v)\cap \mathcal{B}_1 \subset \R^{n}.
\end{align*}  Often subsets of $\R^n$ are embedded in $\R^{n+1}$, as it will be clear from the context. 
$F(v)$ is called the free boundary of $v$.

We consider the free boundary problem,

\begin{equation}\label{FB}\begin{cases}
\text{div}(|z|^\beta \nabla g) = 0, \quad \textrm{in $B_1^+(g) ,$}\\
\dfrac{\p g}{\p U}= 1, \quad \textrm{on $F(g)$}, 
\end{cases}\end{equation}
where $\beta=1-2\alpha, 0<\alpha<1$
$$\dfrac{\p g}{\p U}(x_0):=\di\lim_{t \rightarrow 0^+} \frac{g(x_0+t\nu(x_0),0)} {t^\alpha } , \quad \textrm{$X_0=(x_0,0) \in F(g)$}.$$ 
Here $\nu(x_0)$ denotes the unit normal to $F(g)$ at $x_0$ pointing toward $\mathcal{B}_1^+(g)$ and $U$ is the function defined in \eqref{U}. 

\subsection{The solution U}

Recall that $$U(t,z)=h^{2\alpha}, \quad \quad h:=r^{1/2} \cos \frac \theta 2.$$
The function $h$ is harmonic and it is easy to check that it satisfies
$$h_t=\frac {h}{2r}, \quad \quad |\nabla h|=\frac 12 r^{-1/2}, \quad \quad \frac{h_z}{z}=\frac{1}{4rh}.$$
We obtain
$$\triangle U + \beta \frac{U_z}{z} = 2\alpha (2\alpha-1) h^{2 \alpha -2} (|\nabla h|^2-h\frac {h_z}{z})=0,$$
and since $U$ is $C^2$ in its positive set, it is a viscosity solution.

Clearly the $(n+1)$ dimensional function $U(X):=U(x_n,z)$ is a solution with the free boundary $F(U)=\{x_n=0\}$. Notice that $$\frac{U_n}{U}=\frac{U_t}{U}= \frac{\alpha}{r}.$$

\subsection{Viscosity solutions}
We now introduce the notion of viscosity solutions to \eqref{FB}. 
First we need the following standard notion.

\begin{defn}Given $g, v$ continuous, we say that $v$
touches $g$ by below (resp. above) at $X_0 \in B_1$ if $g(X_0)=
v(X_0),$ and
$$g(X) \geq v(X) \quad (\text{resp. $g(X) \leq
v(X)$}) \quad \text{in a neighborhood $O$ of $X_0$.}$$ If
this inequality is strict in $O \setminus \{X_0\}$, we say that
$v$ touches $g$ strictly by below (resp. above).
\end{defn}

\begin{defn}\label{defsub} We say that $v \in C(B_1)$ is a (strict) comparison subsolution to \eqref{FB} if $v$ is a  non-negative function in $B_1$ which is even with respect to $\{z=0\}$, $v$ is $C^2$ in the set where it is positive and it satisfies
\begin{enumerate} \item $\text{div}(|z|^\beta \nabla v) \geq 0$ \quad in $B_1\setminus \{z=0\}$;\\
\item $F(v)$ is $C^2$ and if $x_0 \in F(v)$ we have
$$v (x,z) = a U((x-x_0) \cdot \nu(x_0), z)+ o(|(x-x_0,z)|^{\alpha}), \quad \textrm{as $(x,z) \rightarrow (x_0,0),$}$$ with $$a \geq 1,$$ where $\nu(x_0)$ denotes the unit normal at $x_0$ to $F(v)$ pointing toward $\mathcal{B}_1^+(v);$\\
\item Either $v$ satisfies (i) strictly or $a >1.$
\end{enumerate} 
\end{defn}

Similarly one can define a (strict) comparison supersolution. 

\begin{defn}\label{definition}We say that $g$ is a viscosity solution to \eqref{FB} if $g$ is a  continuous non-negative function in $B_1$ which satisfies
\begin{enumerate} \item $g$ is locally $C^{1,1}$ in $B^+_1(g)$, even with respect to $\{z=0\}$ and solves (in the viscosity sense)
$$\text{div}(|z|^\beta \nabla g) = 0 \quad \text{in $B_1 \setminus \{z=0\}$;}$$ \item Any (strict) comparison subsolution (resp. supersolution) cannot touch $g$ by below (resp. by above) at a point $X_0 = (x_0,0)\in F(g). $\end{enumerate}\end{defn}

\begin{rem}\label{nondiv} Observe that the equation in (i) can be written in the following non-divergence form
$$\triangle g + \beta \frac{g_z}{z} =0.$$ This fact will be used throughout the paper.
\end{rem}

\begin{rem}\label{c11} We notice that in view of Lemma 2.1 in \cite{S},  $g$  satisfies part (i) in Definition \ref{definition} if and only if 
$g$ solves $$\text{div}(|z|^\beta \nabla g) = 0 \quad \text{in $B_1^+(g$),}$$ in the distributional sense. Equivalently, $g$ is a local minimizer in $B_1^+(g)$ to the energy functional 
$$\int |z|^\beta |\nabla g|^2 dX.$$ In view of this remark, we can apply the standard maximum/comparison principle to functions that satisfy part (i) of Definition \ref{definition}. \end{rem}

\begin{rem}\label{rescale} We remark that if $g$ is a viscosity solution to \eqref{FB} in $B_\rho$, then 
\begin{equation}\label{rr}
g_{\rho}(X) = \rho^{-\alpha} g(\rho X), \quad X \in B_1
\end{equation} 
is a viscosity solution to \eqref{FB} in $B_1.$
\end{rem}

We also introduce the notion of viscosity solutions for the fractional Laplace free boundary problem \eqref{ACalpha}  in the Introduction.
\begin{defn}\label{defsubfrac} We say that $u$ is a viscosity solution to \eqref{ACalpha} if $u$ is a  non-negative continuous function in $\Omega$ and it satisfies
\begin{enumerate} 
\item $(-\Delta)^\alpha u = 0$ \quad in $\Omega$;\\
\item at any point $x_0\in F(u) \cap \Omega$ that admits a tangent ball from either the positive set $\{u>0\}$ or from the zero set $\{u=0\}$ we have
$$u (x) = \left( (x-x_0)^\alpha \cdot \nu(x_0)\right)^++ o(|(x-x_0)|^{\alpha}),$$ where $\nu(x_0)$ denotes the unit normal at $x_0$ to $F(u)$ pointing toward $B_1^+(u).$\\
\end{enumerate} 
\end{defn}

\subsection{Expansion at regular points}
In order to explain better the free boundary conditions in the definitions above we recall Lemma 7.5 from \cite{DS1} about the expansion of solutions $g$ to the equation
\begin{equation}\label{geq}
\text{div}(|z|^\beta \nabla g)=0 \quad \mbox{in} \quad B_1^+(g),
\end{equation}
near points on $F(g)$ that have a tangent ball either from the positive side of $g$ or from the zero-side. The proof in \cite{DS1} is for the case $\alpha=1/2$, however it uses only boundary Harnack inequality (see Theorem \ref{bhi}) and it works identically for any $\alpha \in (0,1)$. 

\begin{prop}\label{exp}
Let $g \in C^{\alpha}(B_1)$, $g \ge 0$, satisfy \eqref{geq}. 
If $$0 \in F(g),  \quad \mathcal{B}_{1/2}(1/2 e_n) \subset B_1^+(g),$$
then 
$$g= a U + o(|X|^\alpha), \quad \mbox{for some $a>0$}.$$ 
The same conclusion holds for some $a \ge 0$ if $$\mathcal{B}_{1/2}(-1/2 e_n) \subset \{g=0\}.$$
\end{prop}

Since viscosity solutions have the optimal $C^\alpha$ regularity (see \cite{CafRS}, \cite{DS1}), a consequence of the proposition above is the following 

\begin{cor}
The function $u$ is a viscosity solution to \eqref{ACalpha} if and only if its extension to $\R^{n+1}$ (reflected evenly across $z=0$) is a viscosity solution to \eqref{FB}. 
\end{cor}

\subsection{Flatness assumption} Theorem \ref{mainT} is stated under the flatness assumption of the free boundary $F(g)$. 
As in Lemma 7.9 in \cite{DS1} this implies closeness between the function $g$ and the one-dimensional solution $U.$ 
Precisely we have

\begin{lem}\label{differentassumption} Assume $g$ solves \eqref{FB}. Given any $\eps>0$ there exist $\bar \eps>0$ and 
$\delta>0$ depending on $\eps$ such that if 
\begin{equation*} \{x \in \mathcal{B}_1 : x_n \leq -\bar \eps\} \subset \{x \in \mathcal{B}_1 : g(x,0)=0\} \subset \{x \in \mathcal{B}_1 : x_n \leq \bar \eps \},\end{equation*} then the rescaling $g_\delta$ (see \eqref{rr}) satisfies
$$U(X-\eps e_n) \le g_\delta(X) \le U(X+\eps e_n)  \quad \mbox{in $B_1$.}$$
\end{lem}

In view of Lemma \eqref{differentassumption} we may assume from now on that 
$$U(X-\eps e_n) \le g(X) \le U(X+\eps e_n) \quad \mbox{in $B_1$,}$$
for some $\eps>0$. 

\subsection{Comparison principle}
We state the comparison principle for problem \eqref{FB}, which in view of Remark \ref{c11} holds in this setting as well. Its proof is standard and can be found in \cite{DR}. As an immediate consequence one obtains Corollary \ref{compmon} which is the formulation of the Comparison Principle used in this paper.

\begin{lem}[Comparison Principle] Let $g, v_t \in C(\overline{B}_1)$ be respectively a solution and a family of  subsolutions to \eqref{FB}, $t \in [0,1]$. Assume that
\begin{enumerate}
\item $v_0 \leq g,$ in $\overline{B}_1;$
\item $v_t \leq g$ on $\p B_1$ for all $t \in [0,1];$
\item $v_t < g$ on $\mathcal{F}(v_t)$ which is the boundary in $\p B_1$ of the set $\p \mathcal{B}_1^+(v_t) \cap \p \mathcal{B}_1$, for all $t\in [0,1];$
\item $v_t(x)$ is continuous in $(x,t) \in \overline{B}_1 \times [0,1]$ and $\overline{\mathcal{B}_1^+(v_t)}$ is continuous in the Hausdorff metric.
\end{enumerate}
Then 
\begin{equation*} v_t \leq g \quad \text{in $\overline{B}_1$, for all $t\in[0,1]$.}
\end{equation*}
\end{lem}

\begin{cor} \label{compmon}Let $g$ be a solution to \eqref{FB} and let $v$ be a  subsolution to \eqref{FB} in $B_2$ which is strictly monotone increasing in the $e_n$-direction in $B_2^+(v)$. Call
$$v_t(X):=v(X+ t e_n), \quad X \in B_1.$$
Assume that for $-1 \leq t_0 < t_1\leq 1$
$$v_{t_0} \leq g, \quad \text{in $\overline{B}_1,$}$$ and
$$v_{t_1} \leq g \quad  \text{on $\p B_1,$}  \quad v_{t_1} < g  \quad \text{on $\mathcal{F}(v_{t_1}).$}$$
Then 
\begin{equation*} v_{t_1} \leq g \quad \text{in $\overline{B}_1$.}
\end{equation*}

\end{cor}

\subsection{Harnack inequalities for $A_2$ weights}

The weight involved in our problem, i.e. $w(z)=|z|^\beta$ where $\beta=1-2\alpha$ with $\alpha \in (0,1)$ belongs to the well-known class of $A_2$ functions as defined by Muchenhoupt \cite{M}. Equations in divergence form involving such weights have been studied in a series of papers by Fabes {et al} in \cite{F1,F2,F3}. In the following, we review the results needed for our purposes. 

\begin{thm} [Harnack inequality \cite{F1}] 
Let $u \ge 0$ be a solution of  $$div(|z|^\beta \nabla u)=0 \quad \mbox{in} \quad  
B_{1}\subset\mathbb  R^{n}.$$ Then, $$\sup_{B_{1/2}} u \leq C \inf_{B_{1/2}} u$$
for some constant $C$ depending only on $n$ and $\beta$.
\end{thm}

\begin{thm} [Boundary Harnack principle \cite{F2}] \label{bhi}
Let $\Omega \subset \R^n$ be a Lipschitz domain, $0 \in \p \Omega$. 
Let $u >0 $ and $v$ be solutions of  $$div(|z|^\beta \nabla u)=div(|z|^\beta \nabla v)=0 \quad \mbox {in} \quad B_1 \setminus (\Omega \times \{0\}),$$  
that vanish continuously on $B_1\cap  (\Omega \times\{0\})$ . Then, 
$$\left [\frac v u \right ]_{C^\gamma(B_{1/2})} \le C$$
for some constants $C$, $\gamma$ depending on $n$ and the Lipschitz constant of $\p \Omega$.
\end{thm}

\section{The function $\tilde g$ and the linearized problem} 

In this section we recall the notion of $\eps$-domain variations of a viscosity solution to \eqref{FB}.
We also introduce the linearized problem associated to \eqref{FB}.

\subsection{The function $\tilde g$.}Let $\eps>0$ and let $g$ be a  continuous non-negative function in $\overline{B}_\rho$. 
Let  $$P:= \{X \in \R^{n+1} : x_n \leq 0, z=0\}, \quad L:= \{X \in \R^{n+1}: x_n=0, z=0\}.$$
To each $X \in \R^{n+1} \setminus P$ we associate $\tilde g_\eps(X) \subset \R$ such that 
\begin{equation}\label{deftilde} U(X) = g(X - \eps w e_n), \quad \forall w \in \tilde g_\eps(X).\end{equation} 

We call $\tilde g_\eps$  the $\eps$- domain variation associated to $g$. 
By abuse of notation, from now on  we write $ \tilde g_\eps(X)$ to denote any of the values in this set. As noted in \cite{DR}, 
if g satisfies\begin{equation}\label{flattilde}U(X - \eps e_n) \leq g(X) \leq U(X+\eps e_n) \quad \textrm{in $B_\rho,$}\end{equation} for all $\eps >0$  we can associate to $g$ a possibly multi-valued function $\tilde{g}_\eps$ defined at least on $B_{\rho-\eps} \setminus P$ and taking values in $[-1,1]$ which satisfies
\begin{equation} \label{til}U(X) = g(X - \eps \tilde{g}_\eps(X)e_n).\end{equation} 

Moreover if $g$ is strictly monotone  in the $e_n$-direction in $B^+_\rho(g)$, then $\tilde{g}_\eps$ is single-valued. 

The following comparison principle is proved in \cite{DR} in the case $\alpha=1/2.$ The proof remains still valid as it only involves Corollary \ref{compmon} and some elementary considerations following from the definition of $\tilde g$. \begin{lem}\label{linearcomp}
Let $g, v$ be respectively a solution and a subsolution to \eqref{FB} in $B_2$, with $v$ strictly increasing in the $e_n$-direction in $B_2^+(v).$ Assume that $g$ satisfies the flatness assumption \eqref{flattilde} in $B_2$
for $\eps>0$ small  and that $\tilde v_\eps$ is defined in $B_{2-\eps} \setminus P$ and satisfies  $$|\tilde v_\eps| \leq C.$$
If,
\begin{equation}\label{start}
 \tilde v_\eps + c \leq  \tilde g_\eps \quad \text{in $(B_{3/2} \setminus \overline{B}_{1/2}) \setminus P,$} 
\end{equation} then 
\begin{equation}\label{conclusion}
 \tilde v_\eps + c \leq \tilde g_\eps  \quad \text{in $B_{3/2} \setminus P.$} 
\end{equation} 
\end{lem}

Finally, we recall the following useful fact. Given $\eps>0$ small and a Lipschitz function $\tilde{\varphi}$ defined on $B_{\rho}(\bar X)$,  with values in $[-1,1]$,  there exists a unique function $\varphi_\eps$ defined at least on $B_{\rho-\eps}(\bar X)$ such that \begin{equation}\label{deftilde2} U(X) = \varphi_\eps(X - \eps \tilde{\varphi}(X)e_n), \quad X \in B_\rho(\bar X).\end{equation}
Moreover such function $\varphi_\eps$ is increasing in the $e_n$-direction. 
If 
 $g$ satisfies the flatness assumption \eqref{flattilde} in $B_1$ and $\tilde{\varphi}$ is as above then (say $\rho,\eps<1/4$, $\bar X \in B_{1/2},$)
\begin{equation}\label{gtildeg}\tilde \varphi \leq \tilde g_\eps \quad \text{in $B_\rho(\bar X) \setminus P$} \Rightarrow \varphi_\eps \leq g \quad \text{in $B_{\rho -\eps}(\bar X)$}.\end{equation}

\subsection{The linearized problem.}  
We introduce here the linearized problem associated to \eqref{FB}.  Here and later $U_n$ denotes the $x_n$-derivative of the function $U$ defined in \eqref{U}. 

Given  $w \in C(B_1)$  and  $X_0=(x'_0,0,0) \in B_1 \cap L,$ we call
$$|\nabla_r w |(X_0) := \di\lim_{(x_n,z)\rightarrow (0,0)} \frac{w(x'_0,x_n, z) - w (x'_0,0,0)}{r}, \quad   r^2=x_n^2+z^2 .$$
Once the change of unknowns  \eqref{deftilde} has been done, the linearized problem associated to \eqref{FB} is 
\begin{equation}\label{linear}\begin{cases} \text{div}(|z|^\beta \nabla (U_n w)) = 0, \quad \text{in $B_1 \setminus P,$}\\ |\nabla_r w|=0, \quad \text{on $B_1\cap L$.}\end{cases}\end{equation}

Our notion of viscosity solution for this problem is below. 

\begin{defn}\label{linearsol}We say that $w$ is a solution to \eqref{linear}  if $w \in C^{1,1}_{loc}
(B_1\setminus P)$, $w$ is even with respect to $\{z=0\}$ and it satisfies (in the viscosity sense)
\begin{enumerate}\item $\text{div}(|z|^\beta \nabla (U_n w)) = 0$ \quad in $B_1 \setminus \{z=0\}$;\\ \item Let $\phi$ be continuous around  $X_0=(x'_0,0,0) \in B_1 \cap L$ and satisfy $$\phi(X) = \phi(X_0) + a(X_0)\cdot (x' - x'_0) + b(X_0) r + O(|x'-x'_0|^2 + r^{1+\gamma}), $$ for some $\gamma>0$ and $$b(X_0) \neq 0.$$  If $b(X_0) >0$ then  $\phi$ cannot touch $w$ by below at $X_0$,  and if $b(X_0)< 0$ then $\phi$ cannot touch $w$ by above at $X_0$. \end{enumerate}\end{defn}

In Section 8, we will investigate the regularity of solutions to \eqref{linear} and obtain the following corollary, which we use in the proof of the improvement of flatness.

\begin{cor}\label{LIF} There exists a universal constant $\rho>0$ such that if $w$ solves \eqref{linear} and $|w|\leq 1$ in $B_1, $ $w(0)=0$ then 
$$a_0 \cdot x' - \frac{1}{8} \rho \leq w(X) \leq a_0\cdot x' +\frac{1}{8} \rho \quad \text{in $B_{2\rho}$}$$ for some vector $a_0 \in \R^{n-1.}$

\end{cor}

\section{Harnack Inequality}

This section is devoted to  a Harnack type inequality for solutions to our free boundary problem \eqref{FB}.

\begin{thm}[Harnack inequality]\label{mainH} There exists $\bar \eps > 0$  such that if $g$ solves \eqref{FB} and it satisfies 
\begin{equation}\label{flatH}U(X +\eps a_0 e_n) \leq g(X) \leq U(X+\eps b_0e_n) \quad \textrm{in $B_\rho(X^*), $}\end{equation}with
$$\eps (b_0 - a_0) \leq \bar \eps \rho, $$  then \begin{equation}\label{impr}U(X +\eps a_1 e_n) \leq g(X) \leq U(X+\eps b_1e_n) \quad \textrm{in $B_{\eta \rho}(X^*)$, }\end{equation} with $$a_0 \leq a_1\leq b_1 \leq b_0, \quad (b_1-a_1) \leq (1-\eta)(b_0-a_0),$$ for a small universal constant $\eta$. \end{thm}

Let $g$ be a solution to \eqref{FB} which satisfies $$U(X - \eps e_n) \leq g(X) \leq U(X+\eps e_n) \quad \textrm{in $B_1$.}$$Let $A_\eps$ be the following set
\begin{equation}\label{Aeps} A_{\eps} := \{(X, \tilde g_\eps(X))  : X \in B_{1-\eps} \setminus P\} \subset \R^{n+1} \times [a_0,b_0].\end{equation}
Since $\tilde g_\eps$ may be multivalued, we mean that given $X$ all pairs $(X, \tilde g_\eps(X))$ belong to $A_\eps$ for all possible values of $\tilde g_\eps(X).$ An iterative argument (see \cite{DR}) gives the following corollary of Theorem \ref{mainH}.

\begin{cor} \label{corHI}If  
$$U(X - \eps e_n) \leq g(X) \leq U(X+\eps e_n) \quad \textrm{in $B_1$,}$$ with $\eps \leq \bar \eps/2$, given $m_0>0$ such that $$2\eps (1-\eta)^{m_0} \eta^{-m_0} \leq \bar\eps,$$ then the set $A_\eps \cap (B_{1/2} \times [-1,1])$  is above the graph of a function $y = a_\eps(X)$ and it is below the graph of a function $y = b_\eps(X)$ with
$$ b_\eps - a_\eps \leq 2(1 - \eta)^{m_0-1},$$
and $a_\eps, b_\eps$ having a modulus of continuity bounded by the H\"older function $\alpha t^\beta$ for $\alpha, \beta$  depending only on $\eta$. 
\end{cor}

The proof of Harnack inequality follows as in the case $\alpha=1/2$. The key ingredient is the lemma below.
 
\begin{lem}\label{babyH}
There exists $\bar \eps > 0$ such that for all  $0 < \eps \leq \bar \eps$ if $g$ is a solution to \eqref{FB}  in $B_1$ such that  
\begin{equation}g(X) \geq U(X) \quad \text{in $B_{1/2},$}\end{equation} and at $\bar X  \in B_{1/8}( \frac{1}{4} e_n)$
\begin{equation}\label{Bound} g(\bar X) \geq U(\bar X + \eps e_n), 
\end{equation} then
 \begin{equation}\label{onesideimprov} 
 g(X) \geq U(X + \tau \eps e_n)  \quad \textrm{in $B_{\delta}$},  \end{equation} for universal constants $\tau, \delta.$
 Similarly, if
 \begin{equation*} g(X) \leq U(X) \quad \text{in $B_{1/2}$,}
\end{equation*} and 
$$g(\bar X) \leq U(\bar X - \eps e_n),$$ then
 \begin{equation*} 
 g(X) \leq U(X -\tau\eps e_n) \quad \textrm{in $B_{\delta}$}.\end{equation*}
\end{lem}

A preliminary basic result is the following.\begin{lem}\label{basic} Let $g\geq 0$ be $C_{llc}^{1,1}$ in $B_2^+(g)$ and solve  \eqref{geq} in $B_2 \setminus \{z=0\}$ and let $\bar X = \frac 3 2 e_n.$ Assume that 
$$g \geq U \quad \text{in $B_2$}, \quad g(\bar X) - U(\bar X) \geq \delta_0$$ for some $\delta_0>0$, then
\begin{equation}\label{gU}g \geq (1+c\delta_0) U \quad \text{in $B_{1}$}\end{equation} for a small universal constant $c$.

In particular, for any $0 < \eps < 2$  
\begin{equation}\label{cor1}U(X + \eps e_n) \geq (1+c\eps)U(X) \quad \text{in $B_1$},\end{equation} with $c$ small universal.
\end{lem}

Its proof can be found in \cite{DR} (Lemma 5.1.) It remains valid since Maximum principle, Harnack Inequality, Boundary Harnack Inequality, and monotonicity of $U$ in the $e_n$-direction,  which are all the ingredients of the proof, are still valid. Harmonic functions in that proof are replaced by solutions to 
\begin{equation}\label{div} \text{div}(|z|^\beta \nabla g)=0.\end{equation}

The main tool in the proof of Lemma \ref{babyH}  will be the following family of radial subsolutions.
Let $R>0$ and denote by $$V_R(t,z) = U(t,z)((n-1)\frac{t}{R} + 1 ). $$ Then set
\begin{equation}v_R(X)= V_R(R- \sqrt{|x'|^2+(x_n-R)^2}, z),\end{equation}
that is we obtain the $n+1$-dimensional function $v_R$ by rotating the 2-dimensional function $V_R$ around  $(0,R,z).$ \begin{prop}\label{sub} If $R$ is large enough, the function $v_R(X)$ is a comparison subsolution to \eqref{FB} in $B_2$ which is strictly monotone increasing in the $e_n$-direction in $B_2^+(v_R)$. Moreover, there exists a function $\tilde v_R$ such that 
\begin{equation}\label{eq}
U(X) = v_R(X - \tilde v_R(X)e_n) \quad \text{in $B_1\setminus P,$}
\end{equation}
and
\begin{equation}\label{estvr}
|\tilde v_R(X) - \gamma_R(X)| \leq \frac{C}{R^2} |X|^2, \quad \gamma_R(X)=- \frac{|x'|^2}{2R} + 2(n-1)\frac{x_n r}{R},
\end{equation} with $r= \sqrt{x_n^2+z^2}$ and $C$ universal.

\begin{proof} We divide the proof of this proposition in two steps.

{\bf Step 1. } In this step we show that $v_R$ is a comparison subsolution in $B_2$ which is monotone in the $e_n$-direction. 

First we see that $v_R$ is a strict subsolution to \eqref{div} in $B_2 \setminus \{z=0\}$. One can easily compute that on such set,

$$\Delta v_R (X) +\beta \frac{(v_R)z(X)}{z}= \Delta_{t,z}V_R(R-\rho, z) - \frac{n-1}{\rho}\p_t V_R(R-\rho, z) + \beta \frac{\p_zV_R (R-\rho,z)}{z}, $$
where for simplicity we call $$\rho :=  \sqrt{|x'|^2+(x_n-R)^2}.$$ 
Also for $(t,z)$ outside the set  $\{(t,0) : t \leq 0\} $
\begin{align*}\Delta_{t,z} V_R (t,z) +& \beta \frac{(V_R)_z(t,z)}{z}= (\p_{tt}+ \p_{zz})V_R(t,z) + \beta \frac{(V_R)_z(t,z)}{z}\\& = \frac{2(n-1)}{R} \p_t U(t,z)+ (1+(n-1)\frac{t}{R})(\Delta_{t,z} U(t,z)+\beta \frac{U_z(t,z)}{z})\\
&= \frac{2(n-1)}{R} \p_t U(t,z),\end{align*}
and 
\begin{equation}\label{pV}\p_t V_R(t,z) = (1+(n-1)\frac{t}{R})\p_t U(t,z) + \frac{n-1}{R} U(t,z).\end{equation}
Thus to show that $v_R$ solves \eqref{div} in $B_2 \setminus \{z=0\}$ we need to prove that in such set
$$\frac{2(n-1)}{R} \p_t U - \frac{n-1}{\rho}[(1+(n-1)\frac{R-\rho}{R})\p_t U + \frac{n-1}{R} U] \geq 0,$$
where $U$ and $\p_t U$ are evaluated at $(R-\rho,z).$ 

Set $ t = R -\rho$, then straightforward computations reduce the inequality above to
$$(n-1)[2(R- t) - R - (n-1)^2 t]\p_t U(t, z) - (n-1)^2U(t, z) \geq 0.$$
Using that $\p_t U( t, z)= \alpha U( t,z)/ r$ with $r^2= t^2 + z^2$, this inequality becomes
$$R \geq 2 t + (n-1)^2 t + \frac{(n-1)}{\alpha} r.$$
This last inequality is easily satisfied for $R$ large enough, since $t, r \leq 3.$

Now we prove that $v_R$ satisfies the free boundary condition in Definition \ref{defsub}.
First observe that $$F(v_R) = \p \mathcal{B}_R(Re_n,0) \cap \mathcal{B}_2,$$ 
and hence it is smooth. By the radial symmetry it is enough to show that the free boundary condition is satisfied at $0 \in F(v_R)$ that is
\begin{equation}\label{expa}v_R(x,z) = a U(x_n,z) + o(|(x,z)|^{\alpha}), \quad \text{as $(x,z) \rightarrow (0,0),$}\end{equation}
with $a \geq 1.$

First notice since $U$ is Holder continuous with exponent $\alpha$, it follows from the formula for $V_R$ that
$$|V_R(t,z) - V_R(t_0,z)| \leq C |t-t_0|^{\alpha} \quad \text{for $|t-t_0| \leq 1.$}$$
Thus for $(x,z) \in B_s,$ $s$ small
$$|v_R(x,z) - V_R(x_n,z)| = |V_R(R-\rho, z) - V_R(x_n, z)| \leq C |R-\rho - x_n|^{\alpha} \leq C s^{2\alpha},$$
where we have used that (recall that $\rho :=  \sqrt{|x'|^2+(x_n-R)^2}$) 
\begin{equation}\label{easy}R - \rho - x_n = - \frac{|x'|^2}{R-x_n + \rho}.\end{equation}
It follows that for $(x,z) \in B_s$
\begin{align*}|v_R(x,z) - U(x_n,z)| &\leq |v_R(x,z) - V_R(x_n, z)| + |V_R(x_n,z) - U(x_n, z)|\\ & \leq Cs^{2\alpha} + |V_R(x_n,z) - U(x_n, z)|.\end{align*} Thus from the formula for $V_R$
$$|v_R(x,z) - U(x_n,z)| \leq Cs^{2\alpha} +(n-1) \frac{|x_n|}{R}U(x_n,z) \leq C' s^{2\alpha}, \quad (x,z) \in B_s$$
which gives the desired expansion \eqref{expa} with $a=1.$

Now, we show that $v_R$ is strictly monotone increasing in the $e_n$-direction in $B_2^+(v_R)$. Outside of its zero plate,
$$\p _{x_n} v_R(x) = - \frac{x_n-R}{\rho} \p_t V_R(R-\rho, z).$$ Thus we only need to show that $V_R(t,z)$ is strictly monotone increasing in $t$ outside $\{(t,0) : t \leq 0\}$ . This follows immediately from \eqref{pV} and the formula for $U$.

\

{\bf Step 2.} In this step we state the existence of $\tilde v_R$ satisfying \eqref{eq} and \eqref{estvr}. Since we have a precise formula for $v_R$ in terms of $U$, this is only a matter of straightforward (though tedious) computations which are carried on in \cite{DR}. 
Also, one needs to use Boundary Harnack inequality for $U$ and its derivatives, the fact that $U$ is homogeneous of degree $\alpha$ and that the ratio $U_t/U= \alpha / r$ (with $\alpha=1/2$ in \cite{DR}.) All these are still valid in this context.  

\end{proof}

\end{prop}

Then, one easily obtain the following Corollary.

\begin{cor}\label{corest}There exist $\delta, c_0,C_0, C_1$ universal constants, such that 
\begin{equation}
\label{2} v_R(X+ \frac{c_0}{R}e_n) \leq (1+\frac{C_0}{R}) U(X), \quad \text{in $\overline{B}_1 \setminus B_{1/4}$},\end{equation} with strict inequality on $F(v_R(X+ \frac{c_0}{R} e_n)) \cap  \overline{B}_1 \setminus B_{1/4},$
\begin{align}
\label{4}& v_R(X + \frac{c_0}{R}e_n) \geq U(X + \frac{c_0}{2R} e_n), \quad \text{in $B_{\delta},$}\\
\label{3}& v_R(X - \frac{C_1}{R} e_n) \leq U(X), \quad \text{in $\overline{B}_1.$}
\end{align}
\end{cor}

\

We are now ready to present the proof of Lemma \ref{babyH}.

\

\textit{Proof of Lemma $\ref{babyH}.$} We prove the first statement.
In view of \eqref{Bound}
$$g(\bar X) - U(\bar X) \geq U(\bar X+\eps e_n) - U(\bar X) = \p_tU(\bar X+ \lambda e_n) \eps \geq c\eps, \quad \lambda \in (0,\eps).$$ From Lemma \ref{basic} we then get\begin{equation} \label{gU2}g(X) \geq (1+ c'\eps) U(X) \quad \text{in $\overline{B}_{1/4}.$}\end{equation} Now let 
$$R= \frac{C_0}{c'\eps},$$
where from now on the $C_i, c_i$ are the constants in Corollary  \ref{corest}. Then, for $\eps$ small enough $v_R$ is a subsolution to \eqref{FB} in $B_2$ which is monotone increasing in the $e_n$- direction and it also satisfies \eqref{2}--\eqref{3}. 
We now wish to apply the Comparison Principle as stated in Corollary \ref{compmon}. Let
$$v_R^t(X) = v_R(X+t e_n), \quad X \in B_1,$$ then according to \eqref{3}, $$v_R^{t_0} \leq U \leq g \quad \text{in $\overline{B}_{1/4}$, with $t_0=-C_1/R.$}$$ Moreover, from \eqref{2} and \eqref{gU2} we get that for our choice of $R$,
$$v_R^{t_1} \leq (1+c'\eps) U \leq g \quad \text{on $\p B_{1/4}$, with $t_1= c_0/R,$}$$ with strict inequality on $F(v_R^{t_1}) \cap \p B_{1/4}.$ 
In particular 
$$g > 0 \quad \text{on $\mathcal{F}(v_R^{t_1})$ in $\p B_{1/4}$}.$$ Thus we can apply  Corollary \ref{compmon} in the ball $B_{1/4}$ to obtain 
$$v_R^{t_1} \leq g \quad \text{in $B_{1/4}$.}$$
From \eqref{4} we have that 
$$U(X+\frac{c_1}{R}e_n) \leq v_R^{t_1}(X) \leq g (X) \quad \text{on $B_{\delta}$}$$ which is the desired claim \eqref{onesideimprov} with $\tau= \frac{c_1 c'}{C_0}$. 
\qed

\section{Improvement of flatness.}

In this section we state the improvement of flatness property for solutions to \eqref{FB} and we provide its proof.  Our main Theorem \ref{mainT} follows from the Theorem below and Lemma \ref{differentassumption}.
\begin{thm}[Improvement of flatness]\label{iflat}There exist $\bar \eps > 0$ and $\rho>0$ universal constants such that for all  $0 < \eps \leq \bar \eps$ if $g$ solves \eqref{FB}  with $0 \in F(g)$ and it satisfies 
\begin{equation}\label{flatimp}U(X - \eps e_n) \leq g(X) \leq U(X+\eps e_n) \quad \textrm{in $B_1$,}\end{equation} then 
\begin{equation}\label{flatimp2}U(x \cdot \nu  - \frac \eps 2 \rho, z) \leq g(X) \leq U(x\cdot \nu+\frac \eps 2 \rho, z) \quad \textrm{in $B_\rho$},\end{equation} for some direction $\nu \in \R^n, |\nu|=1.$ 
\end{thm}

The proof of Theorem \ref{iflat} is divided into the next four lemmas. 

The following Lemma is contained in \cite{DR} (Lemma 7.2) and its proof remains unchanged, since it does not depend on the particular equation satisfied by $g$ but only on elementary considerations related to the definition of $\tilde g_\eps$.

\begin{lem} \label{seclem}Let $g$ be a solution to \eqref{FB} with $0\in F(g)$ and satisfying \eqref{flatimp}. Assume that the corresponding  $\tilde g_\eps$ satisfies 
\begin{equation}\label{bound}  a_0 \cdot x' - \frac{1}{4} \rho\leq  \tilde g_\eps(X) \leq a_0 \cdot x' + \frac{1}{4} \rho \quad \text{in $B_{2\rho} \setminus P,$}\end{equation} for some $a_0 \in \R^{n-1}$. Then if $\eps \leq \bar \eps (a_0,\rho)$  $g$ satisfies \eqref{flatimp2} in $B_\rho$.
\end{lem}

The next lemma follows immediately from the Corollary \ref{corHI} to Harnack inequality. 

\begin{lem}\label{ginfty} Let $\eps_k \rightarrow 0$ and let $g_k$ be a sequence of solutions to \eqref{FB} with $0 \in F(g_k)$ satisfying \begin{equation}\label{flatimp_k}U(X - \eps_k e_n) \leq g_k(X) \leq U(X+\eps_k e_n) \quad \textrm{in $B_1$.}\end{equation}
Denote by  $\tilde g_k$ the  $\eps_k$-domain variation of $g_k$.
Then the sequence of sets 
$$A_k := \{(X, \tilde g_k (X)) : X \in B_{1-\eps_k} \setminus P\},$$ 
has a subsequence that converge uniformly (in Hausdorff distance) in $B_{1/2} \setminus P$ to the graph $$A_\infty := \{(X,\tilde g_\infty(X)) : X \in B_{1/2} \setminus P\},$$ where $\tilde g_\infty$ is  a Holder continuous function.
\end{lem} 

From here on $\tilde g_\infty$ will denote the function from Lemma \ref{ginfty}.

\begin{lem} The limiting function satisfies $\tilde g_\infty \in C^{1,1}_{loc} (B_{1/2} \setminus P).$
\end{lem}

\begin{proof}
We fix a point $Y \in B_{1/2} \setminus P$, and let $\delta$ be the distance from $Y$ to $L$. It suffices to show that the functions $\tilde g_\eps$ are uniformly $C^{1,1}$ in $B_{\delta/8}(Y)$. Indeed , since $g_\eps -U$ is an even function that solves the extension problem in $B_{\delta/2}(Y)$ we find 
$$ \|g_\eps -U\|_{C^{1,1}(B_{\delta/4})(Y)} \le C \|g_\eps -U\|_{L^\infty (B_{\delta /2}(Y)) } \le C \eps,$$
and, by implicit function theorem it follows that
$$\|\tilde g_\eps \|_{C^{1,1}(B_{\delta/8})(Y)} \le C.$$
 Here the constants above depend on $Y$ and $\delta$ as well.
 \end{proof}
 
\begin{lem} \label{limitsol}The function $\tilde g_\infty$ satisfies the linearized problem \eqref{linear} in $B_{1/2}$.
\end{lem} 
\begin{proof}
We start by showing that $U_n \tilde g_\infty$ satisfies \eqref{div} in $B_{1/2} \setminus \{z=0\}. $ 

Let $\tilde \varphi$ be a $C^2$ function which touches $\tilde g_\infty$ strictly by below at $X_0=(x_0,z_0) \in B_{1/2} \setminus \{z=0\}.$ We need to show that 
\begin{equation}\label{des} \Delta (U_n \tilde \varphi)(X_0) + \beta \frac{(U_n \tilde \varphi)_z(X_0)}{z_0} \leq 0.
\end{equation}
Since by Lemma \ref{ginfty}, the sequence $A_k$ converges uniformly to $A_\infty$ in $B_{1/2} \setminus P$ we conclude that there exist a sequence of constants $c_k \rightarrow 0$ and a sequence of points $X_k \in B_{1/2} \setminus \{z=0\}$, $X_k \rightarrow X_0$ such that $\tilde \varphi_k := \tilde \varphi + c_k$ touches $\tilde g_k$ by below at $X_k $ for all $k$ large enough. 

Define the function $\varphi_k$ by the following identity
\begin{equation}\label{varphi}\varphi_k (X- \eps_k \tilde{\varphi_k}(X) e_n) = U(X). \end{equation}

Then according to \eqref{gtildeg} $\varphi_k$ touches $g_k$ by below at $Y_k = X_k - \eps_k \tilde \varphi_k(X_k)e_n \in B_1 \setminus \{z=0\},$ for $k$ large enough. Thus, since $g_k$ satisfies \eqref{div} in $B_1 \setminus \{z=0\}$ it follows that
\begin{equation}\label{sign}
\Delta \varphi_k(Y_k)  + \beta \frac{(\varphi_k)_{n+1}(Y_k)}{z_k}\leq 0,
\end{equation}
with $z_k$ denoting the $(n+1)$-coordinate of $X_k.$

Let us compute $\Delta \varphi_k (Y_k)$ and $(\varphi_k)_{n+1}(Y_k).$ Since $\tilde \varphi$ is smooth, for any $Y$ in a neighborhood of $Y_k$ we can find a unique $X=X(Y)$ such that
\begin{equation}\label{Y}Y= X- \eps_k \tilde \varphi_k (X) e_n.\end{equation}
Thus \eqref{varphi} reads 
$$\varphi_k (Y) = U(X(Y)),$$with $Y_i=X_i$ if $i\neq n$ and $$\frac{\p X_j}{\p Y_i} = \delta_{ij}, \quad \textrm{when $j\neq n$}.$$ 
Using these identities we can compute that

\begin{equation}\label{Deltavarphi}\Delta \varphi_k (Y) = U_n(X) \Delta X_n(Y) + \sum_{j \neq n} (U_{jj}(X) + 2 U_{jn}(X) \frac{\p X_n}{\p Y_j} )+ U_{nn}(X) |\nabla X_n|^2(Y).\end{equation}

From \eqref{Y} we have that
$$D_X Y = I - \eps_k D_X (\tilde \varphi_k e_n).$$
Thus, since $\tilde \varphi_k = \tilde \varphi +c_k$
$$D_Y X = I + \eps_k  D_X (\tilde \varphi e_n)+ O(\eps_k^2), $$
with a constant depending only on the $C^2$-norm of $\tilde \varphi.$

It follows that 

\begin{equation}\label{pxn} \frac{\p X_n}{\p Y_j} = \delta_{jn} + \eps_k \p_j\tilde \varphi (X) + O(\eps_k^2). \end{equation}

Hence

\begin{equation}\label{nablaxn}|\nabla X_n|^2 (Y)= 1+ 2\eps_k \p_n \tilde \varphi (X)+ O(\eps_k^2), \end{equation}

and also, 

$$\frac{\p^2 X_n}{\p Y_j^2} = \eps_k \sum_{i} \p_{ji} \tilde \varphi \frac{\p X_i}{\p Y_j}+ O(\eps_k^2)= \eps_k \sum_{i \neq n} \p_{ji} \tilde \varphi \delta_{ij} + \eps_k \p_{jn}\tilde \varphi \frac{\p X_n}{\p Y_j} + O(\eps_k^2),$$
from which we obtain that
\begin{equation}\label{deltaxn}\Delta X_n = \eps_k \Delta \tilde \varphi + O(\eps_k^2).\end{equation}

Combining \eqref{Deltavarphi} with \eqref{nablaxn} and \eqref{deltaxn} we get that

$$\Delta \varphi_k (Y) =\Delta U(X) +\eps_k U_n \Delta \tilde \varphi  +2\eps_k \nabla \tilde \varphi \cdot \nabla U_n + O(\eps_k^2)(U_n(X) + U_{nn}(X)).$$

From the computations above it also follows that,

\begin{equation*} (\varphi_k)_{n+1}(Y) = U_n(X) \frac{\p X_n}{\p Y_{n+1}} + U_z(X)  \frac{\p X_{n+1}}{\p Y_{n+1}} = U_n(X)(\eps_k \p_{n+1} \tilde \varphi(X) + O(\eps_k^2)) + U_z(X).
\end{equation*}

Using \eqref{sign} together with the fact that $U$ solves \eqref{div} at $X_k$ we conclude that 

$$0 \geq  \Delta (U_n \tilde \varphi) (X_k) +\beta \frac{(U_n \tilde \varphi)z(X_k)}{z_k}+ O(\eps_k)(U_n(X_k) + \beta \frac{U_n(X_k)}{z_k}+ U_{nn}(X_k)).$$

The desired inequality \eqref{des} follows by letting $k \rightarrow +\infty.$

Next we need to show that $$|\nabla_r \tilde g_\infty |(X_0)= 0, \quad X_0=(x'_0,0,0) \in B_{1/2} \cap L,$$ in the viscosity sense of Definition \ref{linearsol}. 
The proof is the same as in the case $\alpha=1/2$, once the properties of the function $v_R$ defined in Proposition \ref{sub} have been established. For convenience of the reader, we present the details.

Assume by contradiction that there exists a function $\phi$ which  touches $\tilde g_\infty$ by below at $X_0=(x'_0,0,0) \in B_{1/2} \cap L$ and such that $$\phi(X) = \phi(X_0) + a(X_0)\cdot (x' - x'_0) + b(X_0) r + O(|x'-x'_0|^2 + r^{1+\gamma}), $$ for some $\gamma>0$, with $$b(X_0) >0.$$   

Then we can find constants $\alpha,  \delta, \bar r$  and a point $Y'=(y'_0,0,0) \in B_2$ depending on $\phi$ such that the polynomial

$$q(X)=\phi(X_0) - \frac{\alpha}{2}|x'-y'_0|^2 + 2\alpha (n-1)x_n r$$
 touches $\phi$ by below at $X_0$  in a tubular neighborhood  $N_{\bar r}= \{|x'-x'_0|\leq \bar r, r \leq \bar r\}$ of $X_0,$ with 
 
 $$\phi- q \geq \delta>0, \quad \text{on $N_{\bar r} \setminus N_{\bar r/2}$.}$$ This implies that
\begin{equation}\label{second}
\tilde g_\infty - q \geq \delta>0, \quad \text{on $N_{\bar r} \setminus N_{\bar r/2}$,}\end{equation}
and 
\begin{equation}\label{third}
\tilde g_\infty (X_0)- q(X_0) =0.\end{equation}
In particular,\begin{equation}\label{thirdprime}
|\tilde g_\infty (X_k)- q(X_k)| \rightarrow 0, \quad X_k \in N_{\bar r} \setminus P, X_k \rightarrow X_0. \end{equation}

Now, let us choose $R_k=1/(\alpha \eps_k)$ and let us define
$$w_{k}(X) = v_{R_k}(X-Y'+\eps_k \phi(X_0) e_n), \quad Y'=(y'_0,0,0),$$
with $v_R $ the function defined in Proposition \ref{sub}. Then the $\eps_k$-domain variation of $w_k$, which we call $\tilde w_k$,  can be easily computed from the definition
$$w_k(X - \eps_k \tilde w_k(X)e_n)=U(X).$$
Indeed, since $U$ is constant in the $x'$-direction, this identity is equivalent to
$$v_{R_k}(X-Y'+\eps_k \phi(X_0) e_n - \eps_k \tilde w_k(X)e_n) = U(X-Y'),$$
which in view of Proposition \ref{sub} gives us
$$\tilde v_{R_k}(X-Y') = \eps_k(\tilde w_k(X) -\phi(X_0)).$$
From the choice of $R_k$, the formula for $q$ and \eqref{estvr}, we then conclude that
$$\tilde w_k (X) = q(X) + \alpha^2 \eps_k O(|X-Y'|^2),$$ and hence
\begin{equation}\label{first} |\tilde w_k - q| \leq C\eps_k \quad \text{in $N_{\bar r} \setminus P.$}\end{equation}
Thus, from the uniform convergence of $A_k$ to $A_\infty$ and \eqref{second}-\eqref{first} we get that for all $k$ large enough
\begin{equation}\label{fcont}
\tilde g_k - \tilde w_k \geq \frac \delta 2 \quad \text{in $(N_{\bar r} \setminus N_{\bar r/2}) \setminus P.$} 
\end{equation}
Similarly, from the uniform convergence of $A_k$ to $A_\infty$ and \eqref{first}-\eqref{thirdprime} we get that for $k$ large
\begin{equation}\label{scont}
\tilde g_k(X_k) - \tilde w_k(X_k) \leq \frac \delta 4,  \quad \text{ for some sequence $X_k \in N_{\bar r} \setminus P, X_k \rightarrow X_0.$}\end{equation} 

On the other hand, it follows from Lemma \ref{linearcomp} and \eqref{fcont} that
\begin{equation*}
\tilde g_k - \tilde w_k \geq \frac \delta 2 \quad \text{in $N_{\bar r} \setminus P,$} 
\end{equation*} which contradicts \eqref{scont}.
\end{proof}

The main Theorem now follows combining all of the lemmas above with the regularity result for the linearized problem, as in the case $\alpha=1/2$. For completeness we present the details.
 
 \
 
 \textit{Proof of Theorem \ref{iflat}.} 
Let $\rho$ be the universal constant from Lemma  \ref{LIF} and assume by contradiction that we
can find a sequence $\eps_k \rightarrow 0$ and a sequence $g_k$ of
solutions to \eqref{FB} in $B_1$ such that $g_k$ satisfies \eqref{flatimp},
i.e.
 \begin{equation}\label{flatimp_k2}U(X - \eps_k e_n) \leq g_k(X) \leq U(X+\eps_k e_n) \quad \textrm{in $B_1$,}\end{equation}
 but it does not satisfy the conclusion of the Theorem.
 
Denote by  $\tilde g_k$ the  $\eps_k$-domain variation of $g_k$.
Then by Lemma \ref{ginfty} the sequence of sets
$$A_k := \{(X, \tilde g_k (X)) : X \in B_{1-\eps_k} \setminus P\},$$ 
converges uniformly (up to extracting a subsequence) in $B_{1/2} \setminus P$ to the graph $$A_\infty := \{(X,\tilde g_\infty(X)) : X \in B_{1/2} \setminus P\},$$ where $\tilde g_\infty$ is  a Holder continuous function in $B_{1/2}$. By Lemma \ref{limitsol}, the function $\tilde g_\infty$ solves the linearized problem \eqref{linear} and hence by Corollary \ref{LIF} $\tilde g_\infty$ satisfies 
\begin{equation}\label{bound2}  a_0 \cdot x' - \frac{1}{8} \rho\leq  \tilde g_\infty(X) \leq a_0 \cdot x' + \frac{1}{8} \rho \quad \text{in $B_{2\rho},$}\end{equation} 
 with $a_0 \in \R^{n-1}$.

From the uniform convergence of $A_k$ to $A_\infty$, we get that for all $k$ large enough
\begin{equation}\label{boundnew}  a_0 \cdot x' - \frac{1}{4} \rho\leq  \tilde g_k(X) \leq a_0 \cdot x' + \frac{1}{4} \rho \quad \text{in $B_{2\rho} \setminus P,$}\end{equation}  and hence from Lemma \ref{seclem},  the $g_k$ satisfy the conclusion of our Theorem (for $k$ large). We have thus reached a contradiction.
\qed

\section{The regularity of the linearized problem.}

The purpose of this section is to prove an improvement of flatness result for viscosity solutions to the linearized problem associated to \eqref{FB}, that is
\begin{equation}\label{linear2}\begin{cases} \text{div} (|z|^\beta \nabla(U_n w)) = 0, \quad \text{in $B_1 \setminus P,$}\\ |\nabla_r w|=0, \quad \text{on $B_1\cap L$,}\end{cases}\end{equation}
where we recall that
for  $X_0=(x'_0,0,0) \in B_1 \cap L,$ we set
$$|\nabla_r w| (X_0) := \di\lim_{(x_n,z)\rightarrow (0,0)} \frac{w(x'_0,x_n, z) - w (x'_0,0,0)}{r}, \quad   r^2=x_n^2+z^2 .$$

The following is our main theorem.

\begin{thm}\label{class} Given a boundary data $\bar h \in C(\p B_1), |\bar h| \leq 1,$ which is even with respect to $\{z=0\}$, there exists a unique classical solution $h$ to \eqref{linear2} such that $h \in C(\overline{B}_1)$, $h = \bar h$ on $\p B_1$, $h$ is even with respect to $\{z=0\}$ and it satisfies
\begin{equation}\label{mainh}|h(X) - h(X_0) -  a' \cdot (x'- x'_0)| \leq C (|x'-x'_0|^2 + r^{1+\gamma}), \quad X_0 \in B_{1/2} \cap L,\end{equation} for  universal constants $C, \gamma$ and a vector $a' \in \R^{n-1}$ depending on $X_0.$
\end{thm}

As a corollary of the theorem above we obtain the desired regularity result, as stated also in Section 3.

\begin{thm}[Improvement of flatness]\label{lineimpflat} There exists a universal constant $C$ such that if $w$ is a viscosity solution to \eqref{linear2} in $B_1$ with $$-1 \leq w(X) \leq 1\quad \text{in $B_1,$}$$ then  \begin{equation}\label{boundlin2}  a_0 \cdot x' -C|X|^{1+\gamma} \leq w(X) - w(0)\leq a_0 \cdot x' + C |X|^{1+\gamma},\end{equation}for some vector $a_0\in \R^{n-1}$.\end{thm}

The existence of the classical solution of Theorem \ref{class} will be achieved via a variational approach in the appropriate weighted Sobolev space. The advantage of working in the variational setting is that the difference of two solutions remains a solution. This is not obvious if we work directly with viscosity solutions.

We say that $h \in H^1(U_n^2 dX, B_1)$ is a minimizer to the energy functional 
$$J(h) := \int_{B_1} |z|^\beta U_n^2 |\nabla h|^2 dX,$$
if
$$J(h) \leq J(h+\phi), \quad \forall \phi \in C_0^\infty(B_1).$$
Since $J$ is strictly convex this is equivalent to
$$\di\lim_{\eps \rightarrow 0} \frac{J(h)-J(h+\eps \phi)}{\eps} =0,  \quad \forall \phi \in C_0^\infty(B_1),$$ which is satisfied if and only if
$$\int_{B_1} |z|^\beta U_n^2 \nabla h \cdot \nabla \phi \;dX = 0, \quad  \forall \phi \in C_0^\infty(B_1). $$

\

Below, we briefly describe the relation between minimizers and viscosity solutions. First, a minimizer $h$ solves the equation
$$\textrm{div}(|z|^\beta U_n^2 \nabla h) = 0 \quad \text{in $B_1,$}$$ which in $B_1 \setminus P$ is equivalent to solving  \begin{equation}\label{div_eq}\text{div}(|z|^\beta \nabla(U_n h)) = 0 \quad \text{in $B_1 \setminus P.$}\end{equation}

Indeed, if $\phi \in C_0^\infty(B_1 \setminus P)$ then 
$$\int |z|^\beta U_n^2 \nabla h \nabla (\frac{\phi}{U_n}) dX=0.$$
This implies,
$$\int |z|^\beta(U_n \nabla h \nabla \phi - \nabla h \phi \nabla U_n) dX=0.$$
Hence,
$$\int |z|^\beta (\nabla(U_n h) \nabla\phi - \nabla U_n \nabla(h\phi))=0.$$
The second integral is zero, since $U_n$ is a solution of the equation $\text{div}(|z|^\beta \nabla U_n) =0$. Thus, our conclusion follows.

Moreover, we claim that if  $h \in C(B_1)$ is a solution to \eqref{div_eq}, such that 
\begin{equation}\label{claim}
\di\lim_{r \rightarrow 0} h_r(x',x_n,z) = b(x'),
\end{equation}
with $b(x')$  a continuous function, then $h$ is a minimizer to $J$ in $B_1$ if and only if $b \equiv 0.$

\textit{Proof of the claim.}
By integration by parts and the computation above the identity 
$$\int_{B_1} |z|^\beta U_n^2 \nabla h \cdot \nabla \phi \;dX = 0, \quad  \forall \phi \in C_0^\infty(B_1),$$ is equivalent to the following two conditions
\begin{equation}\label{unvharm}\text{div}(|z|^\beta \nabla (U_n h)) = 0 \quad \text{in $B_1 \setminus P,$}\end{equation}
and
\begin{equation}\label{fbforv}\di\lim_{\delta \rightarrow 0} \int_{\p C_\delta \cap B_1} |z|^\beta U_n^2 \phi \nabla h \cdot \nu d\sigma =0,\end{equation}
where $C_\delta$ is the cylinder $ \{r \leq \delta\}$
and $\nu$ the inward unit normal to $C_\delta.$ 

Here we use that 
$$\lim_{\eps \rightarrow 0}  \int_{\{|z|=\eps\} \cap (B_1 \setminus C_\delta)} |z|^\beta U_n^2 \phi h_{\nu}  d\sigma =0.
$$

Indeed, in the set $\{|z|=\eps\} \cap (B_1 \setminus C_\delta)$ we have, ( for some $C$ independent of $\eps$)
$$U_n \leq C |z|^{1-\beta}, $$
and
$$|\nabla (U_nh)|, |\nabla U_n| \leq C |z|^\beta,$$
from which it follows that
$$|\nabla h| \leq C |z|^{-1}.$$

In conclusion we need to show that \eqref{fbforv} is equivalent to $b(x')=0.$

This follows, after an easy computation showing that 

$$\lim_{\delta \rightarrow 0}\int_{\p C_\delta \cap B_1} |z|^\beta U_n^2 \phi \nabla h \cdot \nu d\sigma = C_{\alpha} \int_L b(x') \phi(x',0,0) d x' $$ 
with 
$$C_{\alpha}= \alpha^2 \int_{-\pi}^{\pi} (\cos \theta)^\beta (\cos \frac{\theta}{2})^{2-2\beta} d\theta.$$

\qed

From the claim it follows that the function 
$$v(X) := -\frac{|x'|^2}{n-1} + 2x_n r,$$
is a minimizer of $J$. Using as comparison functions the translations of the function $v$ above we obtain as in  Lemma \ref{babyH} that minimizers $h$ satisfy Harnack inequality. 

Since our linear problem is invariant under translations in the $x'$-direction, we see that discrete differences  of the form 
$$h(X + \tau) - h(X),$$with $\tau$ in the $x'$-direction are also minimizers. Now by standard arguments we obtain the following regularity result. 

\begin{lem}\label{derivativeH} Let $h$ be a minimizer to $J$ in $B_1$ which is even with respect to $\{z =0\}$. Then $D_{x'}^k h \in C^\gamma(B_{1/2})$ and $$[D_{x'}^k h]_{C^{\gamma}(B_{1/2})} \leq C \|h\|_{L^\infty(B_1)},$$ with $C$ depending on the index $k=(k_1,..,k_{n-1}).$

\end{lem}

We are now ready to prove our main theorem. 

\

\noindent \textit{Proof of Theorem \ref{class}.}  
It suffices to show that minimizers $h$ with smooth boundary data on $\p B_1$ achieve the boundary data continuously and satisfy the conclusion of our theorem. Then the general case follows by approximation. 

First we show that $h$ achieves the boundary data continuously. At points on $\p B_1 \setminus P$ this follows from the continuity of $U_n h$, since $U_n \ne 0$. 

For points $X_0 \in \p B_1 \cap P$ we need to construct local barriers for $h$ which vanish at $X_0$ and are positive in $\bar B_1$ near $X_0$.  If $X_0 \notin L$ then we consider barriers of the form $$z^{1-\beta} W(x)/ U_n$$ with $W$ harmonic in $x$. If $X_0 \in L$ then the barrier is given by  $$(x'-x_0') \cdot x_0'.$$ 

By Lemma \ref{derivativeH} and \eqref{claim}, it remains to prove that 
\begin{equation}\label{exp2} |h(x',x_n,x) - h(x', 0,0) - b(x') r| \leq C r^{1+\gamma}, \quad (x',0,0) \in B_{1/2} \cap L, \end{equation} \begin{equation}\label{hb}|h_r(x',x_n,z) -b(x') | \leq Cr^{\gamma}, \quad (x',0,0) \in B_{1/2} \cap L,\end{equation}with $C$, $\gamma$ universal and $b(x')$ a continuous function.

Indeed, $h$ solves $$\text div(|z|^\beta \nabla(U_n h)) =0 \quad \text{in $B_1 \setminus P$}. $$ Since $U_n$ is independent on $x'$ we can rewrite this equation as
\begin{equation}\label{2dreduction}\text div_{x_n,z}(|z|^\beta \nabla (U_n  h)) = - |z|^\beta U_n \Delta_{x'} h ,\end{equation}
and according to Lemma \ref{derivativeH} we have that 
$$\Delta_{x'} h \in L^\infty (B_{1/2}).$$  
Thus, for each fixed $x'$, we need to investigate the 2-dimensional problem (in the $(t,z)$-variables)
$$\text div(|z|^\beta \nabla (U_t h)) = |z|^\beta U_t f, \quad \text{in $B_{1/2} \setminus \{t \leq 0, z=0\}$}$$
with $f$ bounded. 

After fixing $x'$, say $x'=0$, we may subtract a constant and assume $h(0,0,0)=0.$ Then $U_t h$ is continuous at the origin and coincides with the  solution $H(t,z)$ to the problem
\begin{equation}\label{eqf}
\text div(|z|^\beta \nabla H) = |z|^\beta U_t f, \quad \text{in $B_{1/2} \setminus \{t \leq 0, z=0\},$}
\end{equation}
such that 
$$H = U_t  h  \quad \text{on $\p B_{1/2}$}, \quad H=0 \quad \text{on $B_{1/2} \cap \{t \leq 0, z=0\}.$}$$

The fact that $U_t h=H$ follows from standard arguments by comparing $H-U_th$ with $\pm \eps U_t$ and then letting $\eps \to 0$.

Using that $U$ is a positive solution to the homogenous equation \eqref{eqf} we may apply boundary Harnack estimate (see Remark \ref{fin}) and obtain that $H/U$ is a $C^\gamma$ function in a neighborhood of the origin. Thus 
$$|H - aU| \leq C_0 r^{\gamma}U, \quad r^2=t^2+z^2, \quad \mbox{$C_0$ universal},$$
for some $a\in \R$. Since $U/U_t= r / \alpha $ we obtain \eqref{exp2} with $b=  a / \alpha$.
 
We show that \eqref{hb} follows from \eqref{exp2} and the derivative estimates for the extension equation. Indeed, the function $\bar H:=H-aU$ above satisfies 
$$|\text div(|z|^\beta \nabla \bar H)| \le C r^{-\alpha}, \quad \quad \|\bar H\|_{L^\infty (B_{2r} \setminus B_r)} \le C r^{\gamma}U,  $$
and the derivative estimates for the rescaled function $\bar H(r(t,z))$ imply
$$|\bar H_r| \le C r^{\gamma -1} U=Cr^\gamma U_t.$$
Using that $$U_t h_r = H_r + (1-\alpha)\frac{H}{r},$$ we easily obtain \eqref{hb}.  

Finally we remark that $b(x')$ is a smooth function since by the translation invariance of our equaltion in the $x'$ direction, the derivatives of $b$ are the corresponding functions in \eqref{exp2} for the derivatives $\partial_{x_i}h, i=1,\ldots, n-1$.

\qed

\begin{rem}\label{fin}
In general boundary Harnack estimate is stated for the quotient $ v / u$ of two solutions (and $u$ positive) to a homogenous equation $L u=0$. 
The result remains valid if $v$ solves the equation $Lv=g$ for a right hand side $g$ that is not too degenerate near the boundary. 
In fact we only need to find an explicit barrier $w$ such that $Lw \ge |g|$ and  $w/u$ is Holder continuous at $0$. 
Then the strategy of trapping $v$ in dyadic balls between multiples $a_k u$ and $b_k u$ can be carried out by trapping $v$ between functions of the type $a_k u +w$ and $b_k u -w$. 

In the case of equation \eqref{eqf} an explicit $w$ is given by $w:=rU$ and it is easy to check that
$$\text div(|z|^\beta \nabla w) \ge c_0 |z|^ \beta U / r,$$
for some positive constant $c_0$.  
\end{rem}

\end{document}